\documentclass[11pt,reqno]{amsart}

\def\volume{\operatorname{vol}}
\def\op{\operatorname}
\usepackage[usenames,dvipsnames]{color}
\usepackage{amsmath,amssymb,bm,amsthm, amsfonts, mathrsfs, eucal, epsfig}
\usepackage{graphicx}
\usepackage{url}
\allowdisplaybreaks

\makeatletter 
\@addtoreset{equation}{section}
\makeatother  

\begin{document}

\newtheorem{Thm}{Theorem}[section]
\newtheorem{Def}{Definition}[section]
\newtheorem{Lem}[Thm]{Lemma}
\newtheorem{Cor}[Thm]{Corollary}
\newtheorem{sublemma}{Sub-Lemma}
\newtheorem{Prop}{Proposition}[section]

\theoremstyle{remark}%
\newtheorem{Rem}{Remark}[section]
\newtheorem{Example}{Example}[section]
\newcommand{\g}[0]{\textmd{g}}
\newcommand{\pr}[0]{\partial_r}
\newcommand{\dif}{\mathrm{d}}
\newcommand{\bg}{\bar{\gamma}}
\newcommand{\md}{\rm{md}}
\newcommand{\cn}{\rm{cn}}
\newcommand{\sn}{\rm{sn}}
\newcommand{\seg}{\mathrm{seg}}

\newcommand{\Ric}{\mbox{Ric}}
\newcommand{\Iso}{\mbox{Iso}}
\newcommand{\ra}{\rightarrow}
\newcommand{\Hess}{\mathrm{Hess}}
\newcommand{\RCD}{\mathsf{RCD}}

\title[Quantitative rigidity of almost maximal volume entropy]{Quantitative rigidity of almost maximal volume entropy for both RCD spaces and integral Ricci curvature bound}

\author{Lina Chen\textsuperscript{*}}
\address[Lina Chen]{Department of mathematics, Nanjing University, Nanjing, P.R.C.}

\email{chenlina\_mail@163.com}
\thanks{\textsuperscript{*} Chen Supported partially by NSFC Grant 12001268.} 

\author{Shicheng Xu\textsuperscript{\dag}}
\address[Shicheng Xu]{Mathematics Department, Capital Normal University, Beijing,
P.R.C.}
\address[Shicheng Xu]{Academy for Multidisciplinary Studies, Capital Normal University, Beijing, P.R.C.}

\email{shichengxu@gmail.com}
\thanks{\textsuperscript{\dag} Xu Supported partially by Beijing NSF Grant No. Z190003 and by NSFC Grant 11821101}

\maketitle

\begin{abstract}

\setlength{\parindent}{10pt} \setlength{\parskip}{1.5ex plus 0.5ex
minus 0.2ex} 
The volume entropy of a compact metric measure space is known to be the exponential growth rate of the measure lifted to its universal cover at infinity. For a compact Riemannian $n$-manifold with a negative lower Ricci curvature bound and a upper diameter bound, it was known that it admits an almost maximal volume entropy if and only if it is diffeomorphic and Gromov-Hausdorff close to a hyperbolic space form. We prove the  quantitative rigidity of almost maximal volume entropy for $\RCD$-spaces with a negative lower Ricci curvature bound and Riemannian manifolds with a negative $L^p$-integral Ricci curvature lower bound.  
\end{abstract}

 \section{Introduction}
 
Sphere theorems are classical results in Riemannian geometry. By Bishop-Gromov volume comparison, a Riemannian $n$-manifold $(M^n,g)$ with Ricci curvature $\ge (n-1)$ admits a volume no more than $n$-sphere $\mathbb{S}^n$ of constant curvature $1$, and equality holds if and only if $(M^n,g)$ is isometric to $\mathbb{S}^n$. By Perelman \cite{Pe}, Colding \cite{Co} and Cheeger-Colding \cite{CC2}, the volume of $(M^n,g)$ is closed to that of $\mathbb{S}^n$ if and only if $(M^n,g)$ is diffeomorphic and Gromov-Hausdorff close to $\mathbb S^n$. 

Similar phenomena also happens for hyperbolic manifolds and the volume entropy. For a compact Riemannian manifold $(M,g)$, its volume entropy is defined to be the exponential growth rate of volume at infinity of the Riemannian universal cover $(\tilde M, \tilde g)$ of $(M,g)$, i.e., 
$$h(M,g)=\lim_{R\to \infty}\frac{\ln\volume(B_R(\tilde x))}{R},$$
where $\tilde x\in \tilde M$ is a fixed point.
When $M$ is compact, the limit always exists and is independent of the choice of $\tilde x$ \cite{Ma}. By direct calculation, for any hyperbolic manifold $\mathbb{H}^n/\Gamma$, $h(\mathbb{H}^n/\Gamma)=n-1$. By Bishop volume comparison, the volume entropy $h(M^n,g)$ of a compact Riemannian $n$-manifold $(M^n,g)$ with Ricci curvature $\ge -(n-1)$ is no more than $n-1$, and by Ledrappier-Wang \cite{LW} (cf. Liu \cite{Liu}) equality holds if and only if $(M^n,g)$ is isometric to a hyperbolic manifold $\mathbb{H}^n/\Gamma$. By Chen-Rong-Xu \cite{CRX}, after fixed a diameter upper bound $\operatorname{diam}(M^n,g)\le D$, $(M^n,g)$ is diffeomorphic and Gromov-Hausdorff close to a hyperbolic manifold if and only if $h(M^n,g)$ is close to $n-1$.  

More general metric spaces with curvature bounds have been extensively studied in the recent three decades, including Alexandrov spaces with curvature bounded below (\cite{BGP}), and $\RCD(K,N)$-spaces (\cite{LV}, \cite{St1, St2}, \cite{AGS}, \cite{Gi13, Gi15}, \cite{EKS}, \cite{AMS19}, \cite{CMi} etc), which are generalizations of sectional curvature lower bound and Ricci curvature lower bound on metric spaces and metric measure spaces respectively. If $(X,d)$ is an Alexandrov $n$-space with curvature $\ge \kappa$, then for any constant $c>0$, $(X,d,c\mathcal{H}^n)$ is a $\RCD((n-1)\kappa,n)$-space, where $\mathcal{H}^n$ is the $n$-dimensional Hausdorff measure of $(X,d)$ (\cite{Pe}, \cite{ZZ}, \cite{AGS}).

The sphere theorems and their quantitative rigidity above are known to hold for $\RCD(n-1,n)$-spaces \cite{PG1}. The maximal and quantitative maximal volume entropy rigidity have also been generalized to Alexandrov spaces in \cite{Ji} and \cite{Ch1} respectively. For a $\RCD(-(n-1), n)$-space $(X,d,\nu)$, its universal cover is known to exist (\cite{MW}). Hence its volume entropy $h(X,d,\nu)$ can be defined similarly on its universal cover $\tilde X$ with the lifted metric $\tilde d$ and measure $\tilde \nu$ (cf. \cite{Rev, BCGS}). By the $\RCD(-(n-1), n)$-curvature condition, $h(X,d,\nu)\le n-1$, and by \cite{CDNPSW}, $h(X,d,\nu)=n-1$ if and only if $(X,d,\nu)$ is isometric to $(\mathbb{H}^n/\Gamma,c\mathcal{H}^n)$ as metric measure spaces.

The quantitative rigidity of almost maximal volume entropy for $\RCD(-(n-1), n)$-spaces earlier are known only for some special cases, i.e., the case that the systole of $X$ has uniform lower bound (\cite{CDNPSW}), i.e., $\inf\{\tilde d(\tilde x, \gamma\tilde x)\,|\, \tilde x\in \tilde X, \gamma\in \bar\pi_1(X)\setminus\{e\}\}\geq l>0$, where $\bar \pi_1(X)$ is the group of deck-transform of $\tilde X$, and the case that $(X,d,\nu)$ is a non-collapsed $\RCD(-(n-1),n)$-space (\cite{Ch1}), i.e., $\nu$ equals to $n$-dimensional Hausdorff measure $\mathcal{H}^n$ of $(X,d)$ with $\mathcal{H}^n(B_1(x))\ge v_0>0$, where $B_1(x)=\{z\in X\,|\, d(x,z)\le 1\}$. 

The first main result in this note is the quantitative rigidity of almost maximal volume entropy for general $\RCD(-(n-1),n)$-spaces. Throughout the paper we use $\Psi(\epsilon|n,D)$ to denote a positive non-decreasing function in $\epsilon$ such that $\Psi(\epsilon|n,D)\to 0$ as $\epsilon\to 0$ and $n, D$ fixed.

\begin{Thm} \label{main-1}
	Given $n>1, D>0$, there exists $\epsilon(n, D)>0$ such that for $0<\epsilon<\epsilon(n, D)$, if a metric measure space $(X, d, \nu)\in \op{RCD}(-(n-1), n)$ satisfies 
	$$h(X, d, \nu)\geq n-1-\epsilon, \quad \op{diam}(X)\leq D,$$
	then $(X,d,\nu)$ is homeomorphic and $\Psi(\epsilon | n, D)$-measured-Gromov-Hausdorff close to a hyperbolic $n$-manifold $(\mathbb{H}^n/\Gamma,c\mathcal{H}^n)$. In particular, up to a renormalizing, $\nu$ coincides with $\mathcal{H}^n$ of $(X,d)$.
	
	Conversely, if $(X,d,\nu)\in \RCD(-(n-1), n)$ is $\epsilon$-measured-Gromov-Hausdorff close to a $(\mathbb{H}^n/\Gamma,\mathcal{H}^n)$, then $|h(X,d,\nu)-(n-1)|\le \Psi(\epsilon|n,D)$.
\end{Thm}

Some classical results for manifolds with lower Ricci curvature bounds are also generalized to  manifolds with integral Ricci curvature lower bound, including the Laplacian comparison, relative volume comparison \cite{PW1, Au1}, almost splitting and almost metric cone rigidity \cite{PW2, TZ, Ch3} etc. Recall that an $n$-manifold $M$  has integral Ricci curvature lower bound if there are constants $p>\frac{n}{2}, R>0, H$ such that
$$\bar k(H ,p, R)=\sup_{x\in M}\left(\frac1{\volume(B_R(x))}\int_{B_R(x)} \rho_H^p dv\right)^{\frac{1}{p}} $$
has an upper bound, where $\rho_H=\max\{-\rho(x)+(n-1)H, 0\}$ and $\rho\left( x\right) $ is the smallest
eigenvalue for the Ricci tensor $\op{Ric} : T_xM\to T_xM$. And if $R=\op{diam}(M)$, we also denote $\bar k(H, p, R)=\bar k(H, p)$.

By \cite{CW}, for an $n$-manifold $(M^n, g)$ with $\op{diam}(M^n)\leq D$ and $\bar k(-1, p)$ small (depends on $n, D, p$),  $h(M^n,g)\le n-1+c(n, p, D)\bar k^{\frac12}(-1, p)$.  It is natural to ask that whether $M^n$ is close to a hyperbolic manifold if $h(M^n,g)$ approaches $n-1$ with integral Ricci curvature bound $\bar k(-1,p)\to 0$.

The second main result is the quantitative rigidity of almost maximal volume entropy for manifolds whose Ricci curvature almost $\ge -(n-1)$ in the $L^p$-integral sense as above.
\begin{Thm} \label{main}
	Given $n>1, D>0, p>\frac{n}{2}$, there exist $\delta(n, D, p), \epsilon(n, D, p)>0$, such that for $0< \delta<\delta(n, D, p)$, $0<\epsilon<\epsilon(n, D, p)$, if a compact $n$-manifold $X$ satisfies that
	$$\op{diam}(X)\leq D, \quad \bar k(-1, p)\leq \delta, \quad h(X)\geq n-1-\epsilon,$$
	then $X$ is diffeomorphic to a hyperbolic $n$-manifold by a $\Psi(\delta, \epsilon | n, D, p)$-isometry.
	
	Conversely, if an $n$-manifold $X$ with $\bar k(-1, p)\leq \delta(n, p, D)$ is $\epsilon$-Gromov-Hausdorff close to a compact hyperbolic $n$-manifold, then $|h(X)-(n-1)|\le \Psi(\epsilon |n, D)$.
	\end{Thm}

Theorem \ref{main} was proved in \cite{Ch2} under the non-collapsing condition, i.e., there is $v>0$ such that $\volume(X)\geq v$.

Since, based on some preliminaries, Theorems \ref{main} and \ref{main-1} will follow from the same arguments, we will only present the detailed proof of Theorem \ref{main}. It is based on the approach of the quantitative rigidity of almost maximal volume entropy for manifolds with lower Ricci curvature bound in \cite{CRX}. In the following we will first recall what have been done for the case of manifolds with lower bounded Ricci curvature, and then point out what is necessarily required in proving Theorems \ref{main} and \ref{main-1}.

Recall in \cite{CRX}, a sequence of $n$-manifolds $X_i$ with 
$$\op{Ric}_{X_i}\geq -(n-1), \quad \op{diam}(X_i)\leq D, \quad h(X_i)\to n-1,$$ was considered, which admits the following equivariant Gromov-Hausdorff convergence:
\begin{equation}\begin{array}[c]{ccc}
(\tilde X_i,\tilde x_i,\Gamma_i)&\xrightarrow{GH}&(\tilde X,\tilde x,G)\\
\downarrow\scriptstyle{\pi_i}&&\downarrow\scriptstyle{\pi}\\
(X_i,x_i)&\xrightarrow{GH} &(X, x), \label{first}
\end{array} \end{equation}
where $\tilde X_i$ is the Riemannian universal cover of $X_i$, and $\Gamma_i=\pi_1(X_i)$ is the fundamental group of $X_i$ acting isometrically on $\tilde X_i$ as deck-transformations. 
By an observation of Liu \cite{Liu} and a generalized version of Cheeger-Colding's ``almost volume cone implies almost metric cone'' (\cite{CRX}, cf. \cite{CC1}), it was shown in \cite{CRX} that $\tilde X_i$ has an almost warped product structure $\mathbb R\times_{e^r} Y$. Then by the  property of the warped product function $e^r$, it is easy to see that, as $r\to \infty$, any ball of a fixed radius centered at $(r,y)$ in $\mathbb R\times_{e^r} Y$ approaches a ball at $(0,y^*)$ in $\mathbb R \times_{e^r} T_yY$, where $T_yY$ is a tangent cone of $y$ in $Y$ and $y^*$ is its vertex. By taking $y$ to be a regular point (i.e., $T_yY=\mathbb R^{k-1}$ for some integer $k\ge 1$), and by the co-compactness of actions by $G$, the warped product $\mathbb R\times_{e^r}\mathbb R^{k-1}$ can be dragged to a fixed point in $\tilde X$. Thus $\tilde X=\mathbb R\times_{e^r}\mathbb R^{k-1}=\mathbb H^k$, where $k\leq n$. 

Let us observe that, if $k=n$ and $G$ acts freely, then the limit space $X$ of $X_i$ is a hyperbolic $n$-manifold.  Then by Cheeger-Colding \cite{CC2}, $X_i$ is diffeomorphic to $X$ and hence the first part of Theorem \ref{main} for manifolds with Ricci curvature $\ge -(n-1)$ is finished. 

In order to show that $k=n$ and $G$ acts freely and discretely, the generalized Margulis lemma by \cite{KW} plays an essential role in \cite{CRX}, which states that the subgroup generated by short loops in $\Gamma_i$ contains a nilpotent subgroup with nilpotent length (step) $\le n$ and index $\leq c(n)$. 

Indeed, by Colding-Naber \cite{CN}, the limit group $G$ is a finite-dimensional Lie group. Then by the generalized Margulis lemma, the connected component $G_0$ of the identity is a connected nilpotent Lie group acting on a hyperbolic space $\mathbb H^k$. On the other hand, it was proved in \cite{CRX} that, if $G_0$ is nilpotent and $\mathbb H^k/G$ is compact, then $G_0$ is either trivial or not nilpotent. Hence the subgroups of $\Gamma_i$ that converges to $G_0=\{e\}$ is finite, which enables one to show that $h(X_i)$ converges to the exponential volume growth rate $k-1$ at infinity on $\mathbb H^k$. Thus $k=n$. At the same time, the discreteness of $G$ implies that the convergence of $X_i$ to $X$ is non-collapsing, and thus by the Reifenberg condition on $\tilde X_i$ it can be seen that the action of $G$ is also free (see \cite[Theorem 2.1]{CRX}). Hence $X$ is isometric to a hyperbolic manifold $\mathbb{H}^n/\Gamma$. 

Note that, provided with recent developments recalled in \S 2 and a generalized Marguls lemma with a uniform index bound, the arguments above can be readily applied to more general spaces, including Alexandrov spaces, $\RCD(-(n-1),n)$-spaces, and manifolds with integral Ricci curvature lower bound.
For example, by the generalized Margulis lemma with a uniform index bound for Alexandrov spaces \cite{XY}, Theorem \ref{main-1} was proved in \cite{Ch2} for Alexandrov spaces by the same approach as above. 

For a sequence of manifolds with integral Ricci curvature lower bound  $\bar k(-1,p)\to 0$, or $\RCD(-(n-1), n)$-spaces, it was proved by the first-named author \cite{Ch1, Ch2} that  the limit space $\tilde X$ of universal covers $\tilde X_i$ is isometric to $\mathbb H^k$ with $1\le k\le n$.
However, the generalized Margulis lemma with \emph{a uniform index bound $C(n)$} is unknown at present for $\RCD(K,n)$-spaces and manifolds with an integral Ricci curvature lower bound, such that the arguments above for $G_0=\{e\}$ fails to apply. Hence, the difficulty is to exclude the collapsing of $\tilde X_i$ and $X$, i.e. to show that $k=n$ and $G$ acts discretely. 

In this note, we will apply a weaker version of the generalized Margulis lemma by \cite{BGT} (see Theorem \ref{mar-lem} below), whose nilpotent subgroup admits \emph{no uniform index bound}. Then $G_0$ would only contain a nilpotent subgroup $N$. We will show in Theorem \ref{prop-hyp} below that $N$ has a fixed point in $\mathbb H^k$, which also enables us to show the convergence of volume entropy such that $k=n$, and $G_0=\{e\}$. 
Thus Theorems \ref{main-1} and \ref{main} are reduced to the non-collapsing case, such that $\tilde X_i$ satisfies the Reifenberg condition, by which $G$ acts freely on $\tilde X$ (the no-collapsing case of Theorems \ref{main-1} and \ref{main} was already settled down in \cite{Ch2} and \cite{Ch3} respectively). 

The main technical Theorem \ref{prop-hyp} is a new result on the limit nilpotent group of deck-transformations, which is based on Besson-Courtois-Gallot \cite[\S 2]{BCG} for the properties of isometric actions on a hyperbolic space, and Chen-Rong-Xu \cite[\S 2]{CRX} for the limit of deck-transformation groups. It also yields a simplified proof in excluding the collapsing of manifolds with lower bounded Ricci curvature and almost maximal volume entropy, which is different from \cite{CRX}.

The authors would like to thank Xiaochun Rong for his helpful remarks in inspiring them to find some new ideas applied in the proof.
 
 \section{Preliminaries}
 In this section, we recall the results earlier known that will be applied in the proof of Theorem~\ref{main} and \ref{main-1}.
 
 \subsection{Manifolds with integral Ricci curvature lower bound}
 
 Given an $n$-manifold $M$, we say that $M$  has integral Ricci curvature lower bound, $(n-1)H$, if $\bar k(H, p, 1)$ has an upper bound, for $p>n/2$ and  $\bar k$ defined as in the introduction. Many properties for manifolds with lower Ricci curvature bound has been generalized to manifolds with integral Ricci curvature bound (see \cite{PW1, PW2, Au1, TZ, Ch3}). In particular, the Laplacian comparison (\cite{PW1, Au1}) and relative volume comparison (\cite{PW1, CW}) holds.  For a compact $n$-manifold $M$ with $\bar k(H, p)\leq c(n, p, \op{diam(M)})$, each of its normal cover, $\hat M$, also satisfies integral Ricci curvature lower bound \cite{Au2}.  And thus the set of manifolds with integral Ricci curvature bound (and their universal covers) is precompact:
\begin{Thm}[\cite{PW1, Au2} Precompactness] \label{compact}
For $n\geq 2, p>\frac{n}{2}, H$, there exists $c(n, p, H)$ such that if a sequence of  compact Riemannian $n$-manifold $M_i$ satisfies that 
$\op{diam}(M)^2\bar{k}_{M_i}(H, p)\leq c(p, n,  H)$, then  there are subsequences of $\{(M_i, x_i)\}$ and $\{(\tilde M_i, \tilde x_i)\}$ that converge in the pointed Gromov-Hausdorff topology where $\tilde M_i$ is the universal cover of $M_i$.
\end{Thm}

For a non-collapsing sequence of manifolds with integral Ricci curvature lower bound, that converges to a Riemannian manifold, the following diffeomorphic stability holds.  
\begin{Lem}[\cite{PW2}]\label{diffe-pro}
Given $n, p>\frac{n}{2}, H$, there is $\delta(n, p, H)>0$, such that if a sequence of compact $n$-manifolds $M_i$ is Gromov-Hausdorff convergent to a compact $n$-manifold $M$ with $\bar k_{M_i}(H, p, 1)\leq \delta(n, p, H)$, then for $i$ large, $M_i$ is diffeomorphic to $M$.
\end{Lem}

In \cite{CW}, it was proved that for a compact $n$-manifold $M$ with $\op{diam}(M)\leq D$, there is $\delta(n, p, D)>0$
such that if $\bar k(-1, p)\leq \delta<\delta(n, p, D)$, then the volume entropy 
$$h(M)\leq n-1- c(n, p, D)\delta^{\frac12}.$$
The quantitative rigidity of almost maximal volume entropy is proved in \cite{Ch2} for the non-collapsing case:

\begin{Thm}[\cite{Ch2}]\label{uni-cov}
Given $n, D, p>\frac{n}{2}$, there exist $\delta(n, D,  p), \epsilon(n, D,  p)>0$, such that for $0< \delta<\delta(n, D,  p)$, $0<\epsilon<\epsilon(n, D, p)$, if a compact $n$-manifold $M$ satisfies that
$$\op{diam}(M)\leq D, \quad \bar k(-1, p)\leq \delta, \quad h(M)\geq n-1-\epsilon,$$
then $\tilde M$ is $\Psi(\delta, \epsilon | n, D, p)$-Gromov-Hausdorff close to a simply connected hyperbolic space form $\Bbb H^k$, $k\leq n$. 

If in addition, there is $v>0$ such that $\volume(M)\geq v$, then $M$ is diffeomorphic and $\Psi(\epsilon, \delta | n, p, D, v)$-Gromov-Hausdorff close to a hyperbolic manifold where $\epsilon, \delta$ may depends on $v$.
\end{Thm}

\subsection{Recent developments about $\RCD(K, n)$-spaces}

For a metric measure space $(X, d, \nu)$, we always assume that the geodesic space $(X, d)$ is complete, separable and locally compact and $\nu$ is a non-negative Radon measure with respect to $d$ and finite on bounded sets. To keep a short presentation, we will skip the definition of $\RCD(K, n)$-spaces here. We refer reader to the survey \cite{Am} for an overview of the definitions and bibliography of $\RCD(K, n)$-spaces. 

Roughly speaking a $\RCD(K, n)$-space is a metric measure space with Ricci curvature bounded below by $K$, dimension bounded above by $n$ and a generalized ``Riemannian structure''.  Any $n$-manifold with $\op{Ric}\geq (n-1)H$, endowed with its length metric and volume is a $\RCD((n-1)H, n)$-space, and the set of $\RCD((n-1)H, n)$-spaces with normalized measure is compact in the measured Gromov-Hausdorff topology  (see \cite[Theorem 7.2]{GMS} and \cite[Theorem 6.11]{AGS}). Thus
any renormalized measured Gromov-Hausdorff limit space $(X, d, \nu)$ of a sequence of $n$-manifolds $(M_i, g_i, \frac{\volume}{\volume(B_1(x_i))})$ with $\op{Ric}_{M_i}\geq (n-1)H$ is a $\RCD((n-1)H, n)$-space. More generally, any $n$-dimensional Alexandrov space with curvature bounded below by $\kappa$ is a $\RCD((n-1)\kappa, n)$-space (\cite{Pe, ZZ, AGS}). A measured Gromov-Hausdorff limit space of a sequence of $n$-manifolds $M_i$ with $\bar k_{M_i}(H, p, 1)\to 0$ is also a $\RCD((n-1)H, n)$-space (\cite{Ke, Ch3}).

Similar to Lemma \ref{diffe-pro}, the following homeomorphic stability holds.
\begin{Lem}[\cite{KM}] \label{home-pro}
	Assume $(M, g)$ is a compact $n$-manifold and $(X_i, d_i, \nu_i)$ is measured Gromov-Hausdorff convergent to $(M, g, \nu)$ with $(X_i, d_i, \nu_i)\in \RCD(K, n)$. Then for $i$ large, $X_i$ is homeomorphic to $M$ by a $\epsilon_i$-Gromov-Hausdorff approximation, $\epsilon_i\to 0$. In particular, $X_i$ is a topological manifold for $i$ large and there is a sequence of positive numbers $c_i\to c$ such that $\nu=c\mathcal H^n$, $\nu_i=c_i\mathcal H^n$.
\end{Lem}

For a connected length space $(X, d)$, a connected covering space $\tilde{\pi}: (\tilde X, \tilde d)\to (X, d)$ is called a universal cover of $(X, d)$ if for any other covering $\pi: (Y, d')\to (X, d)$, there is a covering map $f: \tilde X\to Y$, such that $\pi\circ f=\tilde{\pi}$. Any universal cover is regular and any two universal covers of $(X,d)$ are equivalent to each other (\cite{Sp}). In \cite{MW}, Mondino-Wei showed that any $\RCD(K, N)$-space has a universal cover.
\begin{Thm}[\cite{MW}] \label{uni-exi}
If a metric measure space $(X, d, \nu)\in \RCD(K, n)$, $K\in \Bbb R, n\geq 1$, then $(X, d, \nu)$ has a universal cover space $(\tilde X, \tilde d, \tilde \nu)\in \RCD(K, n)$.
\end{Thm}
In above theorem, one can take $\tilde d, \tilde \nu$ such that $\tilde \pi: \tilde X\to X$ is distance and measure non-increasing and is a local isometry.  The revised fundamental group $\bar \pi_1(X)$ of $X$ is defined to be the deck-transformation group of $\tilde \pi: \tilde X\to X$, which acts on $(\tilde X, \tilde d)$ isometrically and preserves the measure $\tilde \nu$ (see \cite{MW}).  For each $\alpha\in \bar \pi_1(X)$, there is $\gamma\in \pi_1(X, x)$ which induces by curve-lifting a deck-transformation $\Phi(\gamma)$ on $\tilde X$, such that $\alpha=\Phi(\gamma)$, and thus $\Phi: \pi_1(X, x)\to \bar \pi_1(X)$ is a surjective homomorphism.

Similar to Theorem~\ref{uni-cov},  $\RCD(-(n-1), n)$-spaces with an almost maximal volume entropy are known to admit the following properties.
\begin{Thm}[\cite{Ch1}]\label{uni-cov-1}
Given $n>1, D$, there exist $\epsilon(n, D)>0$, such that for $0<\epsilon<\epsilon(n, D)$, if a compact $\RCD(-(n-1), n)$-space $(X, d, \nu)$ satisfies that
$$\op{diam}(X)\leq D,  \quad h(X, d, \nu)\geq n-1-\epsilon,$$
then $\tilde X$ is $\Psi(\epsilon | n, D)$-Gromov-Hausdorff close to $\Bbb H^k$, $k\leq n$.

If in addition, $X$ is non-collapsed, i.e., there is $v>0$ such that $\nu(X)=\mathcal{H}^n(X)\geq v$, then $X$ is $\Psi(\epsilon | n, D, v)$-Gromov-Hausdorff close and homeomorphic to a hyperbolic manifold.
\end{Thm}

\subsection{Generalized Margulis lemma}
A metric space $X$ is said to have bounded packing constant $K>0$, if each ball of radius $4$ in $X$ can be covered by at most $K$ balls of radius 1. By the relative volume comparison \cite{PW1} and \cite{LV, St1, St2}, any manifold with $\bar k(H, p, 1)$ small (depends on $n, p, H$) and  any $\RCD(-(n-1)H, n)$-space have a bounded packing constant $K(H, n)$.

Gromov asked in \cite[\S5.F]{Gr1} that whether finite elements $\gamma_1,\dots, \gamma_q$ in the isometry group of $X$ with a bounded packing constant generate a virtually nilpotent group, if the displacement of $\{\gamma_1,\dots,\gamma_q\}$ is sufficient small at a point $x\in X$? 

This was answered by Breuillard-Green-Tao \cite[Corollary 1.15]{BGT}:  

\begin{Thm}[Generalized Margulis lemma \cite{BGT}]  \label{mar-lem}
	Given $K\geq 1$, there is $\epsilon(K)>0$ such that for a metric space $X$ with packing constant $K$ and an isometric group $\Gamma$ acting discretely on $X$, elements with displacement $\le \epsilon(K)$ generate a virtually nilpotent subgroup, i.e., $$\Gamma_{\epsilon}(x)=\left<\gamma\in \Gamma\,| \, d(x, \gamma(x))<\epsilon(K)\right>$$ contains a finite index nilpotent subgroup with step $\leq C(K)$.
\end{Thm}

It follows from Theorem \ref{mar-lem} that elements of deck-transformations on the universal cover of a $\RCD(-(n-1),n)$-space with displacement $\le  \epsilon(n)$ generates a virtually nilpotent subgroup. 
It happens similarly for compact Riemannian $n$-manifolds with integral Ricci curvature bound $\bar k(-1,p) <\delta(n, p,D)$ and diameter $\le D$.
 
Note that Theorem \ref{mar-lem} is weaker than the traditional generalized Margulis lemma, which was proved for $n$-manifolds with $\Ric\ge-(n-1)$ in \cite{KW} (cf, \cite{FY}, \cite{KPT}), such that the nilpotent subgroup has step $\le n$ and index $\le C(n)$. 

\subsection{Equivariant Gromov-Hausdorff convergence}  In this subsection, we generalize a theorem on the equivariant Gromov-Hausdorff convergence associated to the universal cover by Fukaya-Yamaguchi \cite{FY, FY1}. 
We refer to \cite{FY} for basis properties of equivariant Gromov-Hausdorff topology.

\begin{Thm}[\cite{FY,FY1}]\label{equ-str}
Assume there is a communicate diagram for length metric spaces $(X_i, d_i)$,  $(X, d)$,  $(\tilde X_i,\tilde d_i)$, $(\tilde X, \tilde d)$ with isometric actions by $\Gamma_i$, $G$ respectively: 
\begin{equation}\begin{array}[c]{ccc}
(\tilde X_i,\tilde p_i,\Gamma_i)&\xrightarrow{GH}&(\tilde X,\tilde p,G)\\
\downarrow\scriptstyle{\pi_i}&&\downarrow\scriptstyle{\pi}\\
(X_i,p_i)&\xrightarrow{GH} &(X, p).
\end{array} \end{equation}
If (1) $G_0$ is a normal subgroup of $G$ such that $G/G_0$ is discrete; (2) $\tilde X_i$ is the universal cover of $X_i$; (3) $\Gamma_i$ is the deck-transformations of $\tilde X_i$;  (4) for some $R_0>0$, $G_0$ is generated by $G_0(R_0)=\{g\in G_0\,|\, d(g(\tilde p),\tilde p)\le R_0\}$;  (5) $\tilde X/G$ is compact.
 
Then there are normal subgroups $\Gamma'_i\subset \Gamma_i$ such that 

{\rm (i)} $\Gamma'_i$ is generated by $\Gamma'_i(R_0+\epsilon_i)$, where $\Gamma'_i(R_0+\epsilon_i)$ consists of elements of $\Gamma_i(R_0+\epsilon_i)$ that are $\epsilon_i$-close to $G_0(R_0)$ in the equivariant Gromov-Hausdorff topology, and $\epsilon_i\to 0$;

{\rm (ii)} $(\tilde X_i,  \tilde p_i, \Gamma'_i)$ is equivariant Gromov-Hausdorff convergent to  $(\tilde X, \tilde p, G_0)$, which implies that
$$(\hat X_i=\tilde X_i/\Gamma_i', \hat p_i, \hat \Gamma_i=\Gamma_i/\Gamma_i')\overset{GH}{\longrightarrow}(\hat X=\tilde X/G_0, \hat p, \hat G=G/G_0).$$

{\rm (iii)} for $i$ large, $\hat\Gamma_i$ is isometric to $\hat G$.
\end{Thm}

If $G$ is a Lie group, then $G_0$ in Theorem \ref{equ-str} can be chosen to be the identity component, which is a normal subgroup of $G$. And then $\Gamma_i'=\Gamma_i^\epsilon$, where $\Gamma_i^\epsilon=\left<\gamma\in \Gamma_i\,|\, d(\gamma_i(x),x)\le \epsilon, \forall x\in B_{2R_0}(\tilde p_i)\right>$, and $\epsilon$ is a fixed small positive constant determined by the gap between $G_0$ and other cosets of $G_0$.

Theorem \ref{equ-str} follows the proof of  \cite[Theorem 4.2]{FY1} (for more details see \cite[Theorem 3.1]{Ch1}), by the following observation.

Note that, instead of (2), it was assumed in \cite[Theorem 4.2]{FY1} that 

(2') the universal cover $\tilde X_i$ is simply connected.

We point it out that the proof of \cite[Theorem 4.2]{FY1} still goes through under (2), where the key point is to construct a middle cover of $X_i$.
$$(\hat X_i, \hat p_i, \hat \Gamma_i)\to (\hat X, \hat p, \hat G),$$
such that for $i$ large $\hat \Gamma_i$ is isometric to $\hat G=G/G_0$ and $\pi_1(\hat X_i)=\Gamma'_i$ converges to $G_0$ in the equivariant Gromov-Hausdorff topology.

In \cite{FY,FY1}, based on the pointed equivariant Gromov-Hausdorff convergence $(\tilde X_i,\tilde p_i,\Gamma_i)\to (\tilde X,\tilde p,G)$, a cover $\hat X_i$ of $X_i$ was constructed by gluing the pseudo-cover $V_i(R)=B_R(\tilde p_i)/\Gamma'_i(3R)$ via a group extension of the pseudo-group $\Lambda_i(R)=\Gamma_i(R)/\Gamma'_i(3R)$, where by condition (5) $\op{diam}(X_i)$ is uniformly bounded, and $R>10\op{diam}(X_i)$ is a large but fixed constant such that $V_i(R)$ is a pseudo-cover of $X_i$ (for details of the gluing and group extension, see \cite[Appendix A]{FY}). Since $\tilde X_i$ is the universal cover of $X_i$, the cover $\hat X_i$ constructed by gluing $V_i(R)$ is a middle cover. Since the relation among deck-transformations on $\tilde X_i$, $\hat X_i$ and $\pi_1(X_i,p_i)$ are by lifting of loops, the remaining proof of Theorem \ref{equ-str} are the same as that in \cite[Theorem 4.2]{FY1}, where $\tilde X_i$ is simply connected.

For the case that $(X_i, d_i, \nu_i)\in \RCD(K, n)$, as already considered in \cite[Theorem 3.1]{Ch1}, the universal cover $(\tilde X_i, \tilde d_i, \tilde \nu_i)$ exists (Theorem \ref{uni-exi}). By \cite{GR, So}, $G$ is a Lie group. Hence $G_0$ can be chosen to be its identity component, which is generated by $G_0(\epsilon)$ for some $\epsilon>0$. However, the fact that the isometry group of a $\RCD(K,n)$-space is a Lie group is not used in proving Theorems \ref{main-1} and \ref{main}, due to the fact that $\tilde X=\mathbb{H}^k$ in our proof.

\section{Proofs of the main results}

In this section, we will give the proofs of Theorem~\ref{main} and \ref{main-1}.  
It has a direct relation to the continuity problem of volume entropy.

\subsection{Continuity of volume entropy} 
In general, volume entropy is not continuous under Gromov-Hausdorff convergence. 
\begin{Example}
Let $M^3=T^2\times [0,1]/\phi$, $\phi=\left(\begin{matrix} 1 & 1
\\ 1& 2\end{matrix}\right): T^2\to T^2$. Then the fundamental group $\pi_1(M^3)$ is not nilpotent but solvable. Given any $T^2$-invariant metric $g$, which
splits at each point as $g=g_{T^2}\oplus ds^2$. For any $\epsilon>0$, let
$g_\epsilon=\epsilon^2g|_{T^2}\oplus ds^2$. Then $(M^3,g_\epsilon)\overset{GH}\to(S^1,ds)$ such that the sectional curvature
$|\op{sec}_{g_\epsilon}|\le C$ (a constant). By \v{S}varc-Milnor lemma (\cite{Sva}, \cite{Mi}), $h(M^3, g_{\epsilon})$ is bounded below by the word-length entropy of $\pi_1(M^3)$ multiple a constant $1/(2\op{diam}(M^3,g_\epsilon)+1)$ (\cite{CRX1}). For a solvable but non-nilpotent group, its word-length entropy admits a positive lower bound (\cite[Lemma 3.1]{Os}, depending on the maximal norm of $\phi$'s eigenvalue). It follows that
$h(M^3,g_\epsilon)\ge c>0$, i.e., $h(M^3,g_\epsilon)\nrightarrow h(S^1,ds)=0$. 

\end{Example}
 In \cite{CRX}, for a sequence of $n$-manifolds $M_i$ with a negative lower Ricci curvature bound, it was showed that when $M_i$ converges to a smooth $n$ manifold or when the volume entropy of $M_i$ approaches to its maximal value, the volume entropy is continuous (see \cite[Theorem 0.5 and Theorem 4.6]{CRX}).  

We generalize \cite[Theorem 0.5]{CRX} to the following version. First recall that given a length space $(X, d)$, its $\delta$-cover, $\delta>0$, is defined as a covering space $X_{\delta}$ of $X$ with covering group $\pi_1(X, \delta, x)$ where $\pi_1(X, \delta, x)$ is a normal subgroup of $\pi_1(X, x)$ generated by homotopy classes of closed paths having representative of the form $\alpha^{-1}\beta\alpha$ where $\beta$ is a closed path lying in a ball of radius $\delta$ and $\alpha$ is a path from $x$ to $\beta(0)$.  By \cite[Proposition 3.2]{SW}, if a compact length space $(X, d)$ has a universal cover $\tilde X$, then $\tilde X$ is a $\delta$-cover and for any $0<\delta'<\delta$, $\tilde X=X_{\delta}=X_{\delta'}$.
\begin{Thm} \label{con-pro-2}
Assume two compact metric measure spaces $(X, d, \nu)$ and $(Y, d', \mu)$ are homeomorphic and $\epsilon$-Gromov-Hausdorff close and assume that $(Y, d', \mu)$ has a $\delta$-cover as its universal cover, $\delta>4\epsilon$. Then
$$\left|\frac{h(X, d, \nu)}{h(Y, d',\mu)}-1\right|\leq \Psi(\epsilon | \delta).$$
\end{Thm}
 \begin{proof}
 Since $(X, d, \nu)$ is homeomorphic to $(Y, d', \mu)$, we may view $X$ and $Y$ as one metric space $X$ with two different metrics $d$, $d'$ and two different measures $\nu$, $\mu$, such that the identity map is an $\epsilon$-Gromov-Hausdorff approximation. Thus the universal cover space $\tilde X$ of $(X, d')$ admits two pullback metrics $\tilde d$, $\tilde d'$ and measures $\tilde \nu$, $\tilde \mu$ with the same deck-transformations $\Gamma=\bar\pi_1(X, x)$.  And since $\tilde X$ is a $\delta$-cover of $(X, d')$, by the $\epsilon$-Gromov-Hausdorff closeness, $\tilde X$ is a $\delta/2$-cover of $(X, d)$.
Let $D=\op{diam}(X, d)$, let $B_{R}(\tilde x, \tilde d)$ be  a metric ball of $\tilde x\in (\tilde X, \tilde d)$ of radius $R$. Let
$$\Gamma(R,\tilde d)=\{\gamma\in \Gamma, \, \tilde d(\tilde x, \gamma\tilde x)\leq R\},$$
and let $|\Gamma(R, \tilde d)|$ be the number of $\Gamma(R, \tilde d)$.
Then 
$$\frac{\tilde \nu (B_{R}(\tilde x,\tilde d))}{\tilde \nu (B_{D}(\tilde x,\tilde d))}\le |\Gamma(R,\tilde d)|\le \frac{\tilde \nu (B_{R+D}(\tilde x,\tilde d))}{\nu(B_{D}(x,d))},$$
implies that
\begin{equation}\label{vol-for}
h(X, d, \nu)=\lim_{R\to \infty}\frac{\ln |\Gamma(R,\tilde d)|}{R}. \end{equation}
Now we claim that there is a constant $R_0>0$ such that $$\Gamma(R,\tilde d')\subset  \Gamma\left(R\left(1+\frac{\epsilon}{R_0}\right), \tilde d\right).$$

Assuming the claim, we derive
\begin{eqnarray*}
h(X, d, \nu)&=&\lim_{R\to\infty}\frac{1}{R}\cdot\ln (|\Gamma(R,\tilde d)|)\\
&=&\lim_{R\to\infty}\frac{1}{R\left(1+\frac{\epsilon}{R_0}\right)}\cdot\ln \left(\left|\Gamma\left(R\left(1+\frac{\epsilon}{R_0}\right),\tilde d\right)\right|\right)\\
& \ge& \left(1-\frac{\epsilon}{R_0}\right) \lim_{R\to\infty}\frac{1}{R}\cdot\ln \left(\left|\Gamma\left(R,\tilde d'\right)\right|\right)\\
& = &\left(1-\frac{\epsilon}{R_0}\right)h(Y, d', \mu).
\end{eqnarray*}

In order to prove the claim, it suffices to show that for any $\gamma\in\Gamma$,
$$\frac{d\left(\tilde x, \gamma(\tilde x)\right)}{d'\left(\tilde x, \gamma(\tilde x)\right)}\le 1+\frac{2\epsilon}{\delta},$$
where $R_0=\frac{\delta}{2}$.

Let $\gamma:[0,l]\to (X,d')$ be the geodesic loop in the representation class of $\gamma\in \Gamma$ that has the minimal length.
The identity map $\op{id}:(X, d)\to (X, d')$ is
an $\epsilon$-Gromov-Hausdorff approximation,
for any two points $x,y\in X$, 
$$\left|d(x, y)-d'(x, y)\right|<\epsilon.$$

Take a partition of $0=t_0<t_1<\cdots<t_k=l$ such that for each $0\le j\le k-2$, $|t_{j+1}-t_j|=\frac{\delta}{2}$ and $t_k-t_{k-1}<\frac{\delta}{2}$. Then there are minimal geodesics $\alpha_j$ with respect to $d$ connecting $\gamma(t_j)$ and $\gamma(t_{j+1})$ whose length $\le |t_{j+1}-t_j|+\epsilon$, $0\leq j\leq k-1$. Moreover, the broken geodesic $\alpha$ formed by $\alpha_j$ is represented the same element as $\gamma$ in $\Gamma$ as a deck-transformation of $(\tilde X, \tilde d')$. Because the length of $\alpha$ with respect to $d$ $\le l+\frac{2l}{\delta}\epsilon$, 
$$\frac{d(\tilde x,\gamma(\tilde x))}{d'(\tilde x, \gamma(\tilde x))}\le 1+\frac{2\epsilon}{\delta}.$$

By changing the role of $d'$ and $d$, the above argument also implies 
$$h(Y,  d', \mu)\geq (1-\frac{4\epsilon}{\delta})h(X, d, \nu).$$
 \end{proof}
 
Note that in above theorem, if $Y$ is a smooth Riemannian manifold, then $\delta$ can be taken as the injective radius. 

By \eqref{vol-for}, the volume entropy $h(X,d,\nu)$ is equal to the exponential growth rate of orbits of deck-transformations on $\tilde X$. The convergence of the exponential growth rate of deck-transformations' orbit points on some special normal covers is also necessary in the study of continuity problem of volume entropy. 
  
Inspecting the proof of \cite[Lemma 4.7]{CRX}, the following continuity property holds.
\begin{Lem} \label{con-pro-1}
Assume a sequence of compact metric measure spaces $(X_i, d_i, \nu_i)$ is measured Gromov-Hausdroff convergent to a compact metric measure space $(X, d, \nu)$ and satisfies 
 the following commutative diagram:
\begin{equation*}\begin{array}[c]{ccc}
(\hat X_i, \hat x_i, \hat \Gamma_i) & \xrightarrow{GH} & (\hat X, \hat x, \hat G)\\
\downarrow\scriptstyle{\bar \pi_i}&&\downarrow\scriptstyle{\bar \pi}\\
(X_i,x_i)&\xrightarrow{GH} &(X, x), 
\end{array} \end{equation*}
where $(\hat X_i, \hat d_i, \hat \nu_i)$ is a normal cover of $(X_i, d_i, \nu_i)$, $\hat \Gamma_i$ is the deck-transformations of $\hat X_i$, $\hat \Gamma_i\cong \hat G$. Assume $\hat G$ is discrete and $\hat x$, $x$ are regular points of $\hat X$ and $X$ respectively, where $x\in X$ is called regular if there is an integer $k\ge 0$ such that $(X, r d, x)\overset{GH}\to (\Bbb R^k, 0)$ as $r\to \infty$.
Then
\begin{equation*}\lim_{i\to \infty}\lim_{R\to \infty}\frac{\ln |\hat \Gamma_i(R)|}{R}=h_{\hat x}(\hat X):=\limsup_{R\to \infty}\frac{\ln \hat\nu(B_R(\hat x))}{R}, \label{c-1}\end{equation*}
where $\hat \Gamma_i(R)=\{\hat\gamma\in \hat\Gamma_i, \hat d_i(\hat\gamma_i\hat x_i, \hat x_i)\leq R\}$.
\end{Lem}
\begin{proof}
In fact, by the proof of \cite[Lemma 4.7]{CRX}, using the $\epsilon_i$-equivariant Gromov-Hausdorff approximations
$$(\hat h_i, \hat \phi_i): (B_{\epsilon_i^{-1}}(\hat x_i), \hat x_i, \hat \Gamma_i)\to (\hat X, \hat x, \hat G),$$
where $\epsilon_i\to 0$ and $\hat \phi_i: \hat \Gamma_i\to \hat G$ is an isomorphism,  an $\epsilon_i$-conjugate map 
$$\hat f_i: (\hat X_i, \hat x_i)\to (\hat X, \hat x)$$
can be constructed, i.e., for each $\hat y_i\in \hat X_i$, $\hat \gamma_i\in \hat \Gamma_i$,
\begin{equation}d(\hat f_i(\gamma_i(\hat y_i)), \hat \phi_i(\hat \gamma_i)(\hat f_i(\hat y_i)))\leq \epsilon_i, \label{formula-1}\end{equation}
such that for any $R>500\op{diam}(X)=500D$, 
\begin{equation}\hat f_i: B_R(\hat x_i)\to B_{(1+\epsilon_i/60D)R}(\hat x) \label{formula-2}\end{equation}
is an $\frac{R}{10D}\epsilon_i$-Gromov-Hausdorff approximation. 

Since $\hat x$ and $x$ are regular points, by \cite[Proposition 1.8]{Grove}, $\hat G$ has trivial isotropy group at $\hat x$. And the discreteness of $\hat G$ implies that there is $\delta>0$ such that for any $\hat \gamma, \hat \gamma'\in \hat G$, 
$$\hat d(\hat \gamma\hat x, \hat \gamma'\hat x)>\delta.$$

Assume $\hat{\phi}_i(\hat \gamma_i)=\hat \gamma, \hat \phi_i(\hat \gamma'_i)=\hat \gamma'$. By \eqref{formula-1}, 
$$\hat d(\hat f_i(\hat \gamma_i(\hat x_i)),\hat f_i( \hat \gamma'_i(\hat x_i)))\geq \delta-2\epsilon_i>0,$$
for $i$ large.
Thus
$$|\hat \Gamma_i(R)(\hat x_i)|=|\hat f_i(\hat \Gamma_i(R)(\hat x_i))|=|\hat \phi_i(\hat \Gamma_i(R))(\hat x)|.$$

Now by \eqref{formula-2}, 
$$B_{(1-\epsilon_i/10D)R}(\hat x)\subset \hat f_i(\hat \Gamma_i(R)(\hat x_i))\subset B_{(1+\epsilon_i/10D)R}(\hat x).$$
Then $\hat f_i$ is $\epsilon_i$-conjugate implies
$$\hat G(\hat x)\cap B_{(1-\epsilon_i/10D)R}(\hat x)\subset \hat \phi_i(\hat \Gamma_i(R))(\hat x)\subset \hat G(\hat x)\cap B_{(1+\epsilon_i/10D)R}(\hat x),$$
and thus 
$$|\hat G(\hat x)\cap B_{(1-\epsilon_i/10D)R}(\hat x)|\leq  |\hat \Gamma_i(R)(\hat x_i))|\leq  |\hat G(\hat x)\cap B_{(1+\epsilon_i/10D)R}(\hat x)|.$$
And together with the fact
$$\frac{\hat \nu(B_R(\hat x))}{\nu(B_D(x))}\leq |\hat G(\hat x)\cap B_R(\hat x)|\leq \frac{\hat \nu(B_R(\hat x))}{\hat \nu(B_{\delta}(\hat x))}$$
we derive the conclusion.
\end{proof}

Combing with Lemma~\ref{con-pro-1} and Theorem \ref{equ-str}, a sufficient condition for the continuity of volume entropy can be derived as follows, whose proof will be applied in proving Theorems \ref{main-1} and \ref{main}.

\begin{Thm} \label{con-pro-3}
Assume a sequence of compact metric measure spaces $(X_i, d_i, \nu_i)$  is measured Gromov-Hausdroff convergent to a compact metric measure space $(X, d, \nu)$ with
\begin{equation*}\begin{array}[c]{ccc}
(\tilde X_i,\tilde x_i,\Gamma_i)&\xrightarrow{GH}&(\tilde X,\tilde x,G)\\
\downarrow\scriptstyle{\pi_i}&&\downarrow\scriptstyle{\pi}\\
(X_i,x_i)&\xrightarrow{GH} &(X, x), 
\end{array} \end{equation*}
where $\tilde X_i$ is the universal cover of $X_i$ and $\Gamma_i$ is the deck-transformation group of $\tilde X$. If $G$ is discrete and $\tilde x$, $x$ are regular points, then
$$\lim_{i\to \infty}h(X_i, d_i, \nu_i)=h_{\tilde x}(\tilde X).$$
\end{Thm}
\begin{proof}
Since $G$ itself is discrete, take $G_0=\{e\}$ in Theorem~\ref{equ-str}, and thus we have 
\begin{equation}\begin{array}[c]{ccc}
(\tilde X_i,\tilde x_i,\Gamma_i)&\xrightarrow{GH}&(\tilde X,\tilde x,G)\\
\downarrow\scriptstyle{\hat \pi_i}&&\downarrow\scriptstyle{\hat \pi}\\
(\hat X_i, \hat x_i, \hat \Gamma_i) & \xrightarrow{GH} & (\hat X, \hat x, \hat G)\\
\downarrow\scriptstyle{\bar \pi_i}&&\downarrow\scriptstyle{\bar \pi}\\
(X_i,x_i)&\xrightarrow{GH} &(X, x), \label{equ-1}
\end{array} \end{equation}
where $\hat X_i=\tilde X_i/\Gamma_i^{\epsilon}, \hat X=\tilde X$, $\Gamma_i^{\epsilon}=\Gamma_i'\to \{e\}$, and $\hat \Gamma_i=\Gamma_i/\Gamma_i^{\epsilon}\cong \hat G=G$.

Since $\Gamma_i^{\epsilon}$ is discrete, $G_0=\{e\}$ implies that $|\Gamma_i^{\epsilon}|<C_i<\infty$.

As pointed out in \cite{CRX}, 
\begin{equation}|\hat \Gamma_i(R)|\leq |\Gamma_i(R)|\leq |\Gamma_i^{\epsilon}|\cdot |\hat \Gamma_i(R)|, \label{c-2}\end{equation}
Combing with \eqref{vol-for}, it implies
$$h(X_i, d_i, \nu_i)=\lim_{R\to \infty}\frac{\ln |\hat \Gamma_i(R)|}{R}.$$
Now by Lemma~\ref{con-pro-1} and $\tilde X=\hat X$, 
$$\lim_{i\to\infty}h(X_i, d_i, \nu_i)=h_{\tilde x}(\tilde X).$$
\end{proof}

\begin{Rem}
In general, for a sequence of compact metric measure spaces $(X_i, d_i, \nu_i)$, even if they are all homeomorphic to and measured Gromov-Hausdroff converging to a compact metric measure space $(X, d, \nu)$, the limit group $G$ of deck-transformations $\Gamma_i$ on the universal cover $\tilde X_i$ of $X_i$
is not necessarily discrete, and neither $\tilde X$ is the universal cover of $X$. 

However, by Theorem~\ref{equ-str} and Lemma \ref{con-pro-1}, after taking a suitable normal subgroup $G_0$ of $G$ such that $\hat G=G/G_0$ is discrete (e.g., the identity component of $G$ if $G$ is a Lie group), the exponential growth rate on the orbits always converges in the middle level of \eqref{equ-1} if $\hat x$ and $x$ are regular,
$$\lim_{i\to \infty}\lim_{R\to \infty}\frac{\ln |\hat \Gamma_i(R)|}{R}=h_{\hat x}(\hat X).$$

The key point in the continuity of volume entropy (e.g., in Theorem \ref{con-pro-3}) lies in the connection between $h(X_i)$ and $\lim_{R\to \infty}\frac{\ln |\hat \Gamma_i(R)|}{R}$, where the finiteness of subgroup $\Gamma_i^{\epsilon}$ (e.g., in the case that $G_0$ is trivial) takes a substantial role.   
\end{Rem}

By Theorems \ref{uni-cov} and \ref{uni-cov-1},  the limit $\tilde X$ in \eqref{equ-1} of universal covers $\tilde X_i$ in proving Theorems \ref{main} and \ref{main-1} is hyperbolic space $\mathbb{H}^k$. And it was proved in \cite[Theorem 2.5]{CRX} that if the limit group $G$ on $\tilde X$ is a Lie group that co-compactly acts on a hyperbolic manifold, then its identity component $G_0$ is either trivial or non-nilpotent. For $n$-manifolds with a negative lower Ricci curvature bound, the generalized Margulis lemma \cite{KW} with a uniform index bound implies that $G_0$ is nilpotent, and hence $G$ is discrete. Then Theorem ~\ref{con-pro-3} implies the continuity of volume entropy. 

Since in proving Theorems \ref{main} and \ref{main-1}, we will apply the generalized Margulis lemma (Theorem \ref{mar-lem}) without a uniform index bound, the nilpotent subgroup of $\Gamma_i^\epsilon$ will only converge to a nilpotent subgroup $N$ of $G_0$. Thus a more careful study of $N$ is required.  The main technical part of this note is to derive a property of $N$ (see Theorem \ref{prop-hyp} below), which also implies the finiteness of $\Gamma_i^\epsilon$.

\subsection{Gromov-Hausdorff limit nilpotent isometry groups on $\mathbb{H}^k$}
\begin{Thm} \label{prop-hyp}
Assume the following commutative diagram about geodesic metric spaces:
\begin{equation}\begin{array}[c]{ccc}
(\tilde X_i,\tilde x_i,\Gamma_i)&\xrightarrow{GH}&(\Bbb H^k,\tilde x,G)\\
\downarrow\scriptstyle{\pi_i}&&\downarrow\scriptstyle{\pi}\\
(X_i,x_i)&\xrightarrow{GH} &(X, x), \label{two}
\end{array} \end{equation}
where $k\geq 2$, $X$ is compact, $\tilde X_i$ is the universal cover of $X_i$ and $\Gamma_i$ is the deck-transformations of $\tilde X_i$. 

Assume that $N$ is the limit of a sequence of subgroups $N_i$ of $\Gamma_i$ and that $N$ is a nilpotent normal subgroup of $G_0$.
Then all elements of $N$  are elliptic and have common fixed points in $\Bbb H^k$.
\end{Thm}
Note that by Theorem~\ref{uni-exi}, for $X_i \in \RCD(K, n)$, the universal cover of $X_i$ always exists and is a $\RCD(K, n)$-space. 

As the main geometric technique in this note,  the proof of Theorem~\ref{prop-hyp} is based on some non-trivial properties of isometry group of $\Bbb H^k$, and we will combine them with the Gromov-Hausdorff convergence process. 

Recall that there are three kinds of isometries on $\Bbb H^k$:  elliptic which has fixed point in $\Bbb H^k$, parabolic which has a unique fixed point on the boundary of $\Bbb H^k$ at infinity, $\partial\Bbb H^k$, and hyperbolic which has exactly two fixed points on $\partial\Bbb H^k$ (cf. \cite{Eb}).   Some of the ideas in the proof of Theorem~\ref{prop-hyp} comes from \cite[\S 2]{BCG} and \cite[\S 2]{CRX}.

\begin{proof}[Proof of Theorem~\ref{prop-hyp}]

First we show that $N$ contains no hyperbolic elements.

Assume that there is $\gamma_0\in N$ which is hyperbolic and assume that $\theta, \xi$ are the exactly two fixed points of $\gamma_0$ on the boundary of $\Bbb H^k$ at infinity, $\partial \Bbb H^k$. We claim that for any $\gamma\in N$, $\gamma(\{\theta, \xi\})=\{\theta, \xi\}$. 

Indeed, if there is $\gamma'\in N$ such that $\gamma'(\theta)=\theta'\neq \theta$ or $\xi$, then the hyperbolic element $\gamma_1=\gamma'\gamma_0(\gamma')^{-1}$ satisfies that $\gamma_1(\theta')=\theta'$. By the Ping-pong lemma for cyclic subgroups and a property in \cite[8.1G]{Gr}, we know that a subgroup of $\left<\gamma_0,\gamma_1\right>$ is free which is contradict to that $N$ is nilpotent.

For any $g_0\in G_0$, $g_0^{-1}\gamma_0g_0\in N$ and thus by the claim, $g_0^{-1}\gamma_0g_0(\{\theta, \xi\})=\{\theta, \xi\}$, i.e., 
$$\gamma_0(\{g_0\theta, g_0\xi\})=g_0(\{\theta, \xi\}).$$
Since $\theta, \xi$ are the exactly two fixed points of $\gamma_0$ (or consider the unique axis $a_{\gamma_0}$ of $\gamma_0$ with endpoints $\theta, \xi$ which is preserved by $\gamma_0$, then $\gamma_0(g_0a_{\gamma_0})=g_0a_{\gamma_0}$ implies $g_0a_{\gamma_0}=a_{\gamma_0}$), we have
$$g_0(\{\theta, \xi\})=\{\theta, \xi\}.$$

Furthermore, for any $g\in G$, since $g^{-1}\gamma_0g\in G_0$ and thus $g^{-1}\gamma_0g(\{\theta, \xi\})=\{\theta, \xi\}$, the same argument as above gives that 
$$g(\{\theta, \xi\})=\{\theta, \xi\}.$$

 As $G$ preserves the unique axis $a_{\gamma_0}$ between $\theta$ and $\xi$, as pointed already in the proof of \cite[2.6.2]{CRX}, $G/G_0$ has a fixed point on $\Bbb H^k/G_0$ (the orbit of $a_{\gamma_0}(t)$) and thus is finite. Since $\Bbb H^k/G=(\Bbb H^k/ G_0)/(G/G_0)$ is compact, it follows that $\Bbb H^k/G_0$ must be also compact.
 
 On the other hand, by the standard theory of isometric group actions and the hyperbolic geometry, $\Bbb H^k/G_0$ cannot be compact. Hence a contradiction is derived. 
 
 Indeed, since $\gamma_0$ fixed $\theta$ and $\xi$, and $G_0$ is a connected Lie group, $G_0$ cannot contain elements that permute $\theta$ and $\xi$, i.e., $G_0$ fixes $\theta$ and $\xi$. It follows that $G_0$ only contains hyperbolic and elliptic elements. Then for any elliptic elements $\gamma$, $\gamma$ fixes every points of $a_{\gamma_0}$, which implies that the normal subgroup $E$ consisting of all elliptic elements of $G_0$ is compact. Since in $\mathbb{H}^k$ the normal exponential map of $a_{\gamma_0}$ from its normal bundle $T^\perp a_{\gamma_0}$ is a diffeomorphism on to $\mathbb{H}^k$ globally, $(\mathbb{H}^k, G_0)$ is equivalent to $(T^\perp a_{\gamma_0}, dG_0)$, where $dG_0$ is the action by the differential of elements of $G_0$. At the same time, $(T^\perp a_{\gamma_0}, dG_0)$ is $G_0$-diffeomorphic to the bundle $G_0\times_{(G_0)_y} T^{\perp}_{y}=(G_0\times T^{\perp}_y)/(G_0)_y$, with fiber the normal space $T^{\perp}_y$ to $a_{\gamma_0}$ at $y$, associated with the principle bundle $(G_0)_y\to G_0\to G_0/(G_0)_y$, where $(G_0)_y=E$ is the isotropy group at $y=a_{\gamma_0}(t)$ for some $t$, and any $g_0\in (G_0)_y$ acts right on $G_0$ and left on $T^\perp_y$ by $(g,v)\mapsto (gg_0^{-1},gy)$ (cf. the slice theorem \cite[Lemma 1.1]{Grove} and the discussion below there). Thus $\Bbb H^k/G_0$ is diffeomorphic to $(G_0\times_{(G_0)_y} T^{\perp}_{y})/G_0=T^{\perp}_y/dE$, which is not compact.

Secondly, we prove that $N$ contains no parabolic elements. 

In fact, assume there is a parabolic element $\gamma_0\in N$ and $\theta\in \partial \Bbb H^k$ is the unique fixed point of $\gamma_0$, $\gamma_0(\theta)=\theta$. Then for any $\gamma\in N$, $\gamma(\theta)=\theta$. This is because, if there exists $\gamma'\in N$ with $\gamma'(\theta)=\theta' \neq \theta$, then the parabolic element $\gamma_1=\gamma'\gamma_0(\gamma')^{-1}$ satisfies that
$\gamma_1(\theta')=\theta'$. By the Ping-pong lemma for cyclic subgroups again, $\left<\gamma_0, \gamma_1\right>$ is free, a contradiction to that $N$ is nilpotent.

Now we know that $N$ contains only parabolic and elliptic elements, and there is $\theta\in \partial \mathbb{H}^k$ that is fixed by all elements of $N$. Then every element of $N$ preserves globally each horoshpere centered at $\theta$ (see \cite[Proposition 3.4]{Ba} and the proof of \cite[Lemma 2.5]{BCG}). It follows that for each $\tilde y\in \Bbb H^k$, the $N$-orbit, $N{\tilde y}$, is contained in a horosphere. Note that there is no segment in any horosphere, and thus $N{\tilde y}$ contains no piece of minimal geodesic (or any three different points in $N{\tilde y}$ can not lie in one geodesic).  Now a similar argument as in the proof of \cite[Theorem 2.5]{CRX} gives a contradiction. 

Indeed, let us take $\gamma_i\in \Gamma_i\to \gamma$, a parabolic element in $N$,  under the equivariant Gromov-Hausdorff convergence. Assume $y_i\to y$ is chosen so that $\gamma_i$ is represented by $\beta_i^{-1}*c_i*\beta_i$, where $c_i$ is a closed geodesic at $y_i$ and $\beta_i$ is a minimal geodesic from $x_i$ to $y_i$.  Then the three different  points $\tilde y_i, \gamma_i\tilde y_i, \gamma_i^2\tilde y_i$ are in one geodesic, i.e., $\tilde d_i(\tilde y_i, \gamma_i\tilde y_i)+\tilde d_i(\gamma_i\tilde y_i, \gamma^2_i\tilde y_i)=\tilde d_i(\tilde y_i, \gamma_i^2\tilde y_i)$ which implies
$$2\tilde d(\tilde y, \gamma\tilde y)=\tilde d(\tilde y, \gamma^2\tilde y).$$
It contradicts to that
 $\tilde y, \gamma\tilde y, \gamma^2\tilde y$ are in a horosphere where 
$$2\tilde d(\tilde y, \gamma\tilde y)>\tilde d(\tilde y, \gamma^2\tilde y).$$

Finally we show that $N$ has common fixed points.

The method here is the same as the proof of \cite[Lemma 2.4]{BCG}. Take $\gamma_1\in Z(N)$, the center of $N$, $\gamma_1\neq e$. Since all element of $N$ are elliptic, the set $F_{\gamma_1}\subset \Bbb H^k$ of all fixed points of $\gamma_1$ is non-empty, which must be a totally geodesic submanifold of $\Bbb H^k$ (cf. \cite[1.9.2]{Eb}). Because $\gamma_1\in Z(N)$, for any $\gamma\in N$ we have $\gamma\gamma_1=\gamma_1\gamma$, and hence $\gamma(F_{\gamma_1})=F_{\gamma_1}$.  

Let us consider the restricting of $N$ on  $F_{\gamma_1}$, and $$N_1=N/\{\gamma\in N\,|\,\text{$\gamma$ acts trivially on $F_{\gamma_1}$}\}$$ be the effective quotient group. Then $N_1$ is also nilpotent. For each $\gamma\in N$, because it is elliptic, there is $x\in \mathbb{H}^k$ such that $\gamma x=x$. Let $y\in F_{\gamma_1}$ be the unique projection point of $x$ to $F_{\gamma_1}$ such that $d(x, F_{\gamma_1})=d(x, y)$. Then by
$$d(x, \gamma y)=d(\gamma x, \gamma y)=d(\gamma x, \gamma F_{\gamma_1})=d(x, y),$$
we see that $\gamma y=y$. And thus $N_1$ also contains only elliptic isometries of $F_{\gamma_1}$.  

Since $\gamma_1\neq e$, $\op{dim}(F_{\gamma_1})<\op{dim}(\Bbb H^k)$. If $N_1=\{e\}$, then we are done. If not, let us take $\gamma_2\in Z(N_1)\setminus \{e\}$ and repeat the argument above to derive an elliptically isometric action by a nilpotent group $N_2$ on  the fixed point set $F_{\gamma_2}$ of $\gamma_2$ in $F_{\gamma_1}$. Then $\op{dim}(F_{\gamma_{2}})<\op{dim}(F_{\gamma_1})$. Iterating this process, it will stop in finite $i_0$ steps such that $F_{\gamma_{i_0}}$ contains only one point or $N_{i_0}=\{e\}$, where $F_{i_0}$ is a common fixed points set of $N$.
\end{proof}

\subsection{Proof of Theorem~\ref{main} and \ref{main-1}}

In this subsection, we will finish the proofs of our main results. Since the proof of Theorem~\ref{main} and Theorem~\ref{main-1} are similar, we will only present the details for  Theorem~\ref{main}. 

\begin{proof}[Proof of Theorem~\ref{main}]
	~
	
By the precompactness Theorem~\ref{compact}, let us consider a sequence of $n$-manifolds $X_i\to X$ satisfying 
$$\bar k_{X_i}(-1, p)\leq \delta_i\to 0, \quad \op{diam}(X_i)\leq D, \quad h(X_i)=n-1-\epsilon_i\to n-1.$$
And by Lemma~\ref{diffe-pro} (\cite{PW2}), to prove Theorem~\ref{main}, we only need to show that $X$ is a hyperbolic $n$-manifold.

Applying Theorem~\ref{equ-str} and Theorem~\ref{uni-cov}, we have
\begin{equation*}\begin{array}[c]{ccc}
(\tilde X_i,\tilde x_i,\Gamma_i)&\xrightarrow{GH}&(\mathbb H^k,\tilde x,G)\\
\downarrow\scriptstyle{\hat \pi_i}&&\downarrow\scriptstyle{\hat \pi}\\
(\hat X_i=\tilde X_i/\Gamma_i^\epsilon, \hat x_i, \hat \Gamma_i=\Gamma_i/\Gamma_i^\epsilon) & \xrightarrow{GH} & (\hat X=\mathbb{H}^k/G_0, \hat x, \hat G=G/G_0)\\
\downarrow\scriptstyle{\bar \pi_i}&&\downarrow\scriptstyle{\bar \pi}\\
(X_i,x_i)&\xrightarrow{GH} &(X, x), 
\end{array} \end{equation*}
where $1\leq k\leq n$, $\tilde X_i$ is the universal cover of $X_i$, $\Gamma_i$ is the deck-transformations of $\tilde X_i$, $G_0$ is the identity component of the limit Lie group $G$ of $\Gamma_i$, $\Gamma_i^{\epsilon}=\left<\gamma\in \Gamma_i\,|\, d(\gamma(z),z)\le \epsilon, \forall z\in B_{2R_0}(\tilde x_i)\right>\to G_0$ and $\hat \Gamma_i=\Gamma_i/\Gamma_i^{\epsilon}\cong \hat G=G/G_0$, 

If $2\leq k\leq n$, by Theorem~\ref{mar-lem}, there is $\epsilon>0$ such that $\Gamma_i^{\epsilon}$ contains a finite index nilpotent normal subgroup $N_i$ with step $\leq C(n)$. Assume $N_i\to N\subset G_0$. Then $N$ is a normal subgroup of $G_0$ and is nilpotent. By Theorem~\ref{prop-hyp}, 
 all elements of $N$ are elliptic and have common fixed points. Assume $\tilde y\in \Bbb H^k$ is a fixed point of $N$, i.e., for each $\gamma\in N$, $\gamma\tilde y=\tilde y$. Let $\tilde y_i\in \tilde X_i$ with $\tilde y_i\to \tilde y$.  Then by the orbits' convergence $N_i \tilde y_i\to N\tilde y=\tilde y$ and the discreteness of $N_i$, $N_i$ must be finite. Since $N_i$ has a finite index in $\Gamma_i^{\epsilon}$, we derive that $|\Gamma_i^{\epsilon}|\le C_i<\infty$, for some $C_i$ (maybe $\to \infty$ as $i\to \infty$).

Now as in the proof of Theorem~\ref{con-pro-3} (see also \cite{CRX}), \eqref{c-2} 
holds. Hence by Lemma~\ref{con-pro-1}, 
$$\lim_{i\to\infty}h(X_i)=\lim_{i\to \infty}\lim_{R\to \infty}\frac{\ln |\hat \Gamma_i(R)|}{R}=n-1=h_{\hat x}(\hat X)\leq h_{\tilde x}(\mathbb H^k)=k-1.$$
And thus $k=n$ and $G_0=\{e\}$. Now the sequences of $X_i$ and $\widetilde X_i$ becomes the non-collapsing case. It follows from Theorem~\ref{uni-cov} that $G$ acts freely and discretely on $\mathbb{H}^n$ such that $X$ is an $n$-dimensional hyperbolic manifold $\mathbb H^n/G$.

If $k=1$, then $G_0=\{e\}$ or $G_0=\Bbb R$. For $G_0=\{e\}$, as above, by Theorem~\ref{con-pro-3}, 
$$0<\lim_{i\to\infty}h(X_i)=n-1\leq h_{\tilde x}(\mathbb R)=0,$$
a contradiction. For $G_0=\Bbb R$, $X=\Bbb R/G$ is a point. Then by Theorem~\ref{mar-lem}, $\Gamma_i$ is virtually nilpotent and thus $h(X_i)=0$, a contradiction to $h(X_i)= n-1-\epsilon_i>0$. Above all, $k\neq 1$.

Conversely, let $X_i\overset{GH}\to \mathbb H^n/\Gamma$ be a sequence of $n$-manifolds with $\bar k_{X_i}(-1, p)\leq \delta(n, p, D)$ as in Lemma~\ref{diffe-pro} and $\op{diam}(X_i)\leq D$. 
By Lemma~\ref{diffe-pro} (\cite{PW2}), for $i$ large, $X_i$ is diffeomorphic to $\mathbb H^n/\Gamma$. 

Now by Theorem~\ref{con-pro-2}, 
$$\left|\frac{h(X_i)}{n-1}-1\right|\leq \Psi(\epsilon_i | \delta).$$
 As the discussion below Theorem~\ref{con-pro-2}, we can take $\delta$ as the injective radius of $\Bbb H^n/\Gamma$ which depends on $n, D$. Thus 
$$|h(X_i)-(n-1)|\leq \Psi(\epsilon_i| n,  D).$$
\end{proof}

The proof of Theorem~\ref{main-1} is the same as above just by changing Lemma~\ref{diffe-pro}, Theorem \ref{uni-cov} in the above discussion by Lemma~\ref{home-pro} and Theorem \ref{uni-cov-1} respectively.

\end{document}